\documentclass[a4paper]{amsart}

\usepackage{hyperref}

\usepackage[initials,nobysame,alphabetic,msc-links]{amsrefs}

\usepackage{amssymb}
\usepackage{stmaryrd, slashed} 

\usepackage{amsthm}
\theoremstyle{plain}
\newtheorem{thm}[]{Theorem}
\newtheorem{prop}[]{Proposition}
\newtheorem{lem}[prop]{Lemma}
\newtheorem{cor}[prop]{Corollary}

\theoremstyle{definition}

\theoremstyle{remark}
\newtheorem{rmk}[prop]{Remark}
\newtheorem{example}[prop]{Example}

\newcommand{\ensemble}[1]{\left\{ #1 \right\}}
\newcommand{\suchthat}{\mid}

\newcommand{\absolute}[1]{\left| #1 \right|}
\newcommand{\pairing}[2]{\left \langle #1, #2 \right\rangle}

\newcommand{\Z}{\mathbb{Z}}

\newcommand{\C}{\mathbb{C}}
\newcommand{\R}{\mathbb{R}}

\newcommand{\nat}{\text{(nat)}}

\newcommand{\triv}{\text{triv}}

\newcommand{\Tang}{\mathrm{T}}
\newcommand{\spinbdl}{\mathbb{S}}

\newcommand{\LevCiv}{\mathrm{LC}}

\newcommand{\diracop}{\slashed{\partial}}

\newcommand{\liealg}[1]{\mathfrak{#1}}
\newcommand{\proj}{\text{pr}}
\newcommand{\vbdl}[1]{\mathcal{#1}}
\newcommand{\UniEnv}{\mathcal{U}}

\DeclareMathOperator{\Tr}{Tr}
\DeclareMathOperator{\Ad}{Ad}

\DeclareMathOperator{\Aut}{Aut}
\DeclareMathOperator{\dom}{dom}
\DeclareMathOperator{\ch}{ch}

\DeclareMathOperator{\ind}{ind}

\DeclareMathOperator{\SO}{SO}
\DeclareMathOperator{\SU}{SU}
\DeclareMathOperator{\PU}{PU}
\DeclareMathOperator{\KK}{{\mathit{KK}}}
\DeclareMathOperator{\Cliff}{\mathrm{Cl}_{\C}}
\DeclareMathOperator{\supp}{supp}

\begin{document}

\title[character of projective Dirac operator]{Index character associated to the projective Dirac operator}
\author{Makoto Yamashita}
\date{\today}       
\thanks{This work has been supported by the Marie Curie Research Training Network
MRTN-CT-2006-031962 in Noncommutative Geometry, EU-NCG.}
\email{makotoy@ms.u-tokyo.ac.jp}
\begin{abstract}
 We calculate the equivariant index formula for an infinite dimensional Clifford module canonically associated to any Riemannian manifold.  It encompasses the fractional index formula of the projective Dirac operator by Mathai--Melrose--Singer.
\end{abstract}

\maketitle
\section{Introduction}
\label{sec:introduction}

Mathai, Melrose, and Singer constructed the projective Dirac operator $\diracop^\proj_M$ for arbitrary Riemannian manifold $M$ as a projective pseudo-differential operator in \cite{MR2258800}.  Its index $\ind_a \diracop^\proj_M$ was defined in terms of the integral kernel $I^{\proj}_{\diracop_M}$ of this projective differential operator.  They showed the \textit{fractional index formula} \cite{MR2258800}*{Theorem 2 and Section 9}
\begin{equation}
  \label{eq:mms-frac-index-formula}
  \ind_a \diracop^\proj_M = \int_M \hat{A}(R_M),
\end{equation}
where the integrand of the right hand side is the $\hat{A}$-form associated to the Riemannian curvature of $M$.  The equality was proved via local index calculation of the heat kernel associated to $I^{\proj}_{\diracop_M}$.  It can be easily seen that the above expression might be a non-integral rational number when $M$ is not spin.

The framework of projective pseudo-differential operators ignited various developments of the study of modules and index theory over the bundles of finite dimensional algebras (or the ones whose fiber is the trace class operator algebra) over manifolds \citelist{\cite{MR2140985}\cite{MR2674880}\cite{arXiv:0907.1257}\cite{arXiv:1007.3667}\cite{arXiv:1011.5800}}.

The projective operator $\diracop^\proj_M$ is associated to a certain Clifford module of $M$, but one has to note the difference between \eqref{eq:mms-frac-index-formula} and the usual index formula
\begin{equation}
\label{eq:cliff-mod-ind-thm}
\ind \diracop_E = \int_M \hat{A}(R_M) \ch(E/\spinbdl),
\end{equation}
for a Clifford module $E$ over $M$.  Here $\diracop_E$ is the twisted Dirac operator acting on the sections of $E$ and $\ch(E/\spinbdl)$ is the relative Chern character of $E$.  The presence of the relative Chern character $\ch(E/\spinbdl)$ makes the integral in the right hand side to be an integer.

Let $\pi^\nat \colon \SU(N) \curvearrowright V^\nat$ be the natural representation of $\SU(N)$.  The Levi--Civita connection on $F_{\SO(n)}$ induces a $1$-form $\nabla$ of first order differential operators acting on the space $C^\infty(P, V^\nat)$.  This means, by definition, when $X$ is a vector field over $M$, one obtains a first order differential operator $\nabla_X$ on $P$.

Since $P$ can be regarded as the bundle of trivializations of the Clifford bundle $\Cliff(M)$, the sections of $\Cliff(M)$ correspond to the $\PU(N)$-invariant functions of $P$ into $M_N(\C)$.  Then the space of functions from $P$ into $V^\nat$ is naturally a module over $\Cliff(M)$.  Thus one may define the associated Dirac operator 
\[
\diracop_M = \sum_k c(e_k) \nabla_{e_k}
\]
 on $C^\infty(P, V^\nat)$ for any local frame $(e_i)_{i=1}^n$.

The operator $\diracop_M$ can be regarded as a transversely elliptic operator~\cite{MR0482866} on the $\SU(N)$-manifold $P$.  Hence one obtains a distribution $\ind_{\SU(N)}(\gamma, \diracop_M)$ on $\SU(N)$ which is invariant under conjugation and satisfies
\[
\pairing{\ind_{\SU(N)}(\gamma, \diracop_M)}{\chi^\pi} = \ind(\diracop_M p^\pi)
\]
where $\chi^\pi$ the character of any irreducible representation $\pi$ of $\SU(N)$, and $p^\pi$ is the projector onto the $\pi$-isotypic component of the action $\pi^P \otimes \pi^\nat$ on $C^\infty(P, V^\nat)$, where $\pi^P$ is the action induced by the translation on $P$.

In~\cite{MR2396250}, it was shown that the $\SU(N)$-equivariant operator $\diracop_M$ descends to the projective Dirac operator $\diracop^\proj_M$ on the projective spin bundle over the Azumaya bundle $\Cliff(M)$.  For such `descent' one has
\[
\pairing{\ind_{\SU(N)}(\gamma, \diracop_M)}{\phi} = \ind_a \diracop^\proj_M
\]
when $\phi$ is a smooth function on $\SU(N)$ which is constantly equal to $1$ around $e$ and has small enough support.  As mentioned at the end of the introduction of~\cite{MR2396250}, the above formula gives `the coefficient of the Dirac function' in the distribution $\ind_{\SU(N)}(\gamma, \diracop_M)$.

The aim of this paper is to relate the fractional index formula \eqref{eq:mms-frac-index-formula} to the classical index formula \eqref{eq:cliff-mod-ind-thm} of Clifford modules and give a refined formula (Corollary \ref{cor:ind-distribution-full-formula}) for the contribution of `higher order derivatives of the Dirac function' in the distribution $\ind_{\SU(N)}(\gamma, \diracop_M)$.

The space $L^2(P, V^\nat)$ can be regarded as the space $\Gamma(M, \vbdl{E})$ of the sections of an infinite dimensional Clifford module $\vbdl{E}$ over $M = P/\PU(N)$.  It is defined as the induced vector bundle
\begin{equation}
  \label{eq:equivar-induct-pict}
  \vbdl{E} = P \times_{\PU(N)} L^2(\PU(N)) \otimes V^\nat,
\end{equation}
where we consider the left translation action $\lambda$ on $L^2(\PU(N))$ and the trivial action on $V^\nat$.  The precise meaning of the correspondence between $L^2(P, V^\nat)$ and $\vbdl{E}$ is that we have
\[
L^2(P, V^\nat) \simeq (L^2(P) \otimes  L^2(\PU(N)) \otimes V^\nat)^{\pi^P \otimes \lambda \otimes \triv_{V^\nat}(\PU(N))},
\]
where the right hand side can be regarded as $\Gamma(M, \vbdl{E})$.

The action $\rho \otimes \pi^\nat$ of $\SU(N)$ on $L^2(\PU(N)) \otimes V^\nat$ commutes with the one $\lambda \otimes \triv$ of $\PU(N)$.  Hence it induces an action on \eqref{eq:equivar-induct-pict}.  This action corresponds to the action $\pi^P \otimes \pi^\nat$ on $L^2(P, V^\nat)$.

If one tries to consider an analogue of the equivariant index theorem in the case of $E = \vbdl{E}$, the relative curvature form
$F^{\vbdl{E}/\spinbdl}$ of $\vbdl{E}$ should be locally given by the action of
the curvature of a vector bundle $\vbdl{E}_0$ satisfying $\vbdl{E} \simeq \vbdl{E}_0 \otimes \spinbdl$.  Note that the vector bundle
$\vbdl{E}_0$ has infinite rank, hence the fiberwise trace
$\Tr_{\vbdl{E}_0} e^{-F^{\vbdl{E}/\spinbdl}}$ will become infinite.  To remedy this, we shall construct an action of $\SU(N)$ on $\vbdl{E}_0$ such that its tensor product with the trivial action on $\spinbdl$ is the action on $\vbdl{E}$.  Then an equivariant choice of curvature on $\vbdl{E}_0$ will allow us to compute the trace of $\Tr_{\vbdl{E}_0}( e^{-F^{\vbdl{E}/\spinbdl}} \pi(\phi))$ of the composition of the curvature $e^{-F}$ with the convolution by any auxiliary function $\phi$ on $\SU(N)$.

Thus the expression
\begin{equation}
\label{eq:distr-expr-rel-chern-char-naive}
\ch_{\SU(N)}(\gamma, \vbdl{E}/\spinbdl)_x = \Tr_{\vbdl{E}_0}(\gamma e^{-F^{\vbdl{E}/\spinbdl}})
\end{equation}
for $x$ in $M$ and $\gamma$ in $\SU(N)$ defines a differential form on $M$ of distributions on $\SU(N)$.  Then the analogue of the index formula \eqref{eq:cliff-mod-ind-thm} for the Clifford module $\vbdl{E}$ should be
\begin{equation}
\label{eq:equivar-cliff-mod-ind-formula}
\ind_{\SU(N)}(\gamma, \diracop_M) = \int_M \hat{A}(R_M) \ch_{\SU(N)}(\gamma, \vbdl{E}/\spinbdl).
\end{equation}

\section{Preliminaries}
\label{sec:preliminaries}

Most of the constructions in this section appear in the literature in some way or other.  We just recall their definitions in order to fix the notations and conventions.  For the sake of simplicity we assume that $n = \dim(M)$ is even and put $N = 2^{n/2}$.  Let $P$ be the $PU(N)$ principal bundle over $M$ induced from the frame bundle $F_{\SO(n)}$ by the natural group embedding
\begin{equation}
\label{eq:so-n-pu-N-correspondence}
 \SO(n) \rightarrow \Aut(\Cliff(\R^n)) \simeq \PU(N).
\end{equation}

Let $\sigma$ be a section of $F_{\SO(n)}$ defined on an open set $U$, and $\tilde{\sigma} = (\sigma, e)$ be the section of $P$ defined by $\sigma$ and the constant mapping $M \rightarrow \SU(N), x \mapsto e$.  Under the identification of \eqref{eq:equivar-induct-pict} any section $f$ of $\vbdl{E}$ over $U$ can be expressed using a function $\xi\colon U \ni x \mapsto \xi_x \in L^2(\PU(N), V^\nat)$ by
\begin{equation}
  \label{eq:abstract-section-induced-realization-correspondence}
  f_x \leftrightarrow (\sigma_x, \xi_x).  
\end{equation}
The Clifford module structure on $\vbdl{E}$ can be described as
\begin{equation}
  \label{eq:cliff-mod-str-on-e-defn}
  c(X).(\sigma_x, f_x) = (\sigma_x, \Ad_{g}(\sigma^*_x(c(X))) f_x(g)),
\end{equation}
where $\sigma^*_x$ is the isomorphism $\Cliff(\Tang_x M) \rightarrow M_N(\C)$
given by $\sigma_x \in P_x$.

Let $A^\LevCiv$ be the Levi--Civita connection $1$-form $\Tang F \rightarrow
\liealg{so}(n)$ on $F_{\SO(n)}$.  One obtains the induced connection form
\begin{equation}
\label{eq:universal-curvature-form}
A\colon \Tang P \rightarrow \liealg{su}(N)
\end{equation}
via \eqref{eq:so-n-pu-N-correspondence}.  Hence it is given by the collection of linear maps $T_p P \rightarrow
\liealg{su}(N)$ which are retractions of the embedding $\liealg{su}(N)
\rightarrow T_p P$ coming from the action map $g \mapsto p.g$ from $\SU(N)$ to
$P$, and satisfy the equivariance condition $A_X = \Ad_g A_{g.X}$ for any $X \in
T_p P$ and $g \in G$.  

Given a vector field
$X$ on $M$, define the operator $A^{(\vbdl{E}, \sigma)}_{X}$ on the
sections of $\vbdl{E}$ over $U$ by
\begin{equation}
  \label{eq:invar-conn-action-from-section}
  A^{(\sigma)}_{X}.(\tilde{\sigma}_x, \xi_x) = (\tilde{\sigma}_x, \left(\lambda (A_{d \tilde{\sigma}(X_x)})\otimes 1_{V^\nat} \right) (\xi_x)).
\end{equation}

\begin{lem}
  \label{lem:conn-infin-dimens}
  Let $\psi$ be a function from $U$ to $\SU(N)$.  Then one has
  \[
  A^{(\sigma.\psi)}_X(\sigma.\psi_x, \psi_x^{-1}\xi_x) = (\sigma.\psi_x, \lambda(A_{d(\sigma)(X)})^{\psi_x} \psi_x \xi_x - d\psi(X).\psi_x^{-1}.\xi_x)
  \]
 for any vector field $X$ on $U$.
\end{lem}

By Lemma \ref{lem:conn-infin-dimens}, the covariant derivative
\[
\nabla_X (\sigma_x, \xi_x) = (\sigma_x, X(\xi_x)) + A^{(\sigma)}_X.(\sigma_x, \xi_x))
\]
associated to $A$ does not depend on the choice of $\sigma$.  Moreover the
formula \eqref{eq:cliff-mod-str-on-e-defn} implies $[\nabla_X, c(Y)] =
c(\nabla^{\LevCiv}_X Y)$, which shows that $A$ is a Clifford connection.
Finally, we put
\[
\diracop_M = \sum_{i=1}^n c(e_i) \nabla_{e_i}
\]
for any local frame $(e_i)_{i=1}^n$.

For each irreducible representation $\pi$ of $\SU(N)$, the restriction of $\diracop$ to the $\pi$-isotypic component of $L^2(P, V^\nat)$ becomes a $K$-cycle $ \diracop_\pi$ over $C(P) \rtimes \SU(N)$.  Hence we obtain a family of $K$-cycles $(\diracop_\pi)_{\pi \in \widehat{\SU(N)}}$ parametrized by the irreducible representations of $\SU(N)$.  Since the $\pi$-isotypic component of $L^2(P, V^\nat)$ is trivial unless the central part of $\pi$ agrees with that of the natural representation, our interest is in $(\diracop_\pi)_{\pi \in \widehat{\SU(N)}_{\nat}}$ where $\widehat{\SU(N)}_{\nat}$ is the collection of the irreducible representation classes whose central character agree with that of the natural representation.

For each $\pi \in \widehat{\SU(N)}_{\nat}$, the number $\ind \diracop_\pi$ is finite and agrees with the multiplicity of $\pi$ in the representation $\ind \partial_M$ of $\SU(N)$.  Let $d_\pi$ be the dimension of the representation space of $\pi$, and $\chi^\pi$ be its normalized character
\[
\chi^\pi(\gamma) = \frac{1}{d_\pi} \Tr(\pi(\gamma)).
\]
Then the sum
\[
\ind_{\SU(N)}(\gamma, \diracop_M) = \sum_{\pi \in \widehat{\SU(N)}_{\nat}} ( \ind \diracop_\pi )\chi^{(\pi)}(\gamma)
\]
defines a distribution over $\SU(N)$ invariant under the conjugation.

\subsection{Pairing of \texorpdfstring{$\UniEnv(\mathfrak{g})$}{U(g)} and \texorpdfstring{$C^\infty(G)$}{C(G)}}
\label{sec:trace-pair-invar}

When $a$ is an element of the universal enveloping algebra $\UniEnv(\liealg{su}(N))$ of $\SU(N)$ can be regarded as a distribution on $\SU(N)$ with support $\ensemble{e}$ by
\[
\pairing{a}{\phi} = \left(\lambda(a)\phi\right)(e)
\]
for $\phi \in C^\infty(\SU(N))$, where $\lambda(a)$ is the right invariant differential operator represented by $a$.  The above pairing has another interpretation
\[
\pairing{a}{\phi} = \Tr_{L^2(\SU(N))}(\lambda(a)\lambda(\phi)) = \Tr_{L^2(\SU(N))}(\rho(a)\rho(\phi)).
\]

Let $T$ be a conditional expectation of $\UniEnv(\liealg{su}(N))$ onto its center.  

\begin{lem}
\label{lem:pu-e-nat-decomp-mult-calc}
Let $a$ be an element of $\UniEnv(\liealg{su}(N))$ and $\phi$ an element in $C^\infty(\SU(N))$.  Then one has
\begin{equation}
  \label{eq:env-alg-elem-left-reg-smooth-func-right-reg-tr-pair}
\sum_{\pi \in \widehat{\SU(N)}_\nat} \Tr(\pi(a)) \Tr(\bar{\pi}(\phi)) = \frac{1}{N}\sum_{g \in Z(\SU(N))} \chi^\nat(g) \pairing{T(a)}{\lambda_g\phi}.
\end{equation}
\end{lem}

\begin{proof}
The operator $\pi(T(a))$ is a scalar satisfying $\Tr(\pi(a)) = \Tr(\pi(T(a)))$ for each irreducible representation $\pi$ of $\SU(N)$.  Hence one has
\[
\Tr_{L^2(\SU(N))}(\lambda(a) \rho(\phi)) = \sum_{\pi \in \widehat{\SU(N)}} \Tr(\pi(a)) \Tr(\bar{\pi}(\phi)) = \Tr_{L^2(\SU(N))}(\rho(T(a)) \rho(\phi)).
\]
Combining this with
\[
\Tr_{L^2(\SU(N))_{\chi^\nat}}(\lambda(a) \rho(\phi)) = \frac{1}{N}\sum_{g \in Z(\SU(N))} \chi^\nat(g) \Tr_{L^2(\SU(N))}(\lambda(a) \rho(\phi) \rho_g),
\]
one obtains
\begin{equation*}
  \Tr_{L^2(\SU(N))_{\chi^\nat}}(\lambda(a) \rho(\phi)) = \frac{1}{N} \sum_{g \in Z(\SU(N))} \chi^\nat(g) \pairing{T(a)}{\lambda_g\phi}.
\end{equation*}
This proves the assertion.
\end{proof}

\section{\texorpdfstring{$\KK$}{KK}-element associated to projective Dirac operator}
\label{sec:texorpdfstr-elem-ass}

For each character $\chi$ on the center of $\SU(N)$, let $J_\chi$ denote the closure of
\[
\ensemble{ f \in C(\SU(N), C(P)) \suchthat \forall z \in Z(\SU(N))\colon f(g z) = \chi(z) f(g)}
\]
in $C(P) \rtimes \SU(N)$.  Then the algebra $C(P) \rtimes
\SU(N)$ admits a direct sum decomposition
\[
C(P) \rtimes \SU(N) \simeq \oplus_{\chi \in \widehat{Z(\SU(N)}} J_\chi
\]
by bilateral ideals.
The representation of
$C(P) \rtimes \SU(N)$ on $L^2(P, V^\nat)$ factors through the projection onto
$J_{\nat}$, the factor corresponding to the central character of the natural representation.

The bimodule $L^2(P, V^\nat)$ is a completion of the following $C^*$-module $\vbdl{F}$ over $\Cliff(M)$.  The subspace $C(P, V^\nat)$ of $L^2(P, V^\nat)$ admits a $\Cliff(M)$-valued inner product characterized by
\[
p^*((\xi, \eta)_x) = \int_{\PU(N)} \Ad_g (\eta_p \otimes \xi_p^*) d g
\]
where we identify $p \in P_x$ with an algebra isomorphism $p^*\colon \Cliff(T_x M) \rightarrow M_N(\C)$.  The completion $\vbdl{F}$ of $C(P, V^\nat)$ with respect to the above inner product admits an action of $C(P) \rtimes \SU(N)$ as $\Cliff(M)$-compact operators.

\begin{prop}
  \label{prop:mor-equiv-cross-prod-ideal-cliff-alg}
  The bimodule $\vbdl{F}$ gives a strong Morita equivalence between
  $J_{\nat}$ and $\Cliff(M)$.
\end{prop}

Each irreducible representation $\pi$ of $\SU(N)$ gives a class $[\pi]$ in
$K_0(C^*\SU(N))$.  Under the inclusion $\iota\colon C^*\SU(N) \rightarrow C(P) \rtimes \SU(N)$,
one obtains an element $\iota_*[\pi]$ in $K_0(C(P) \rtimes \SU(N))$.  By Proposition
\ref{prop:mor-equiv-cross-prod-ideal-cliff-alg}, one obtains an element $\iota_*[\pi]$ of $K_0(\Cliff(M))$ given by the Clifford module $V_\pi = p_\pi L^2(P, V^\nat)$.

The operator $\diracop_M$ and the representation of $C(P) \rtimes \SU(N)$ on $L^2(P, V^\nat)$ defines a spectral triple over $C(P) \rtimes \SU(N)$.  Consequently we obtain the associated element $\alpha$ of $\KK(C(M)
\rtimes \SU(N), \C)$ represented by the phase of $\diracop_M$ as in~\cite{MR1769535}, and the map
\[
\ind_{\diracop_M}\colon K_0(C(P) \rtimes \SU(N)) \simeq K_0(\Cliff(M)) \rightarrow \Z.
\]

\begin{example}
  \label{ex:nat-comp-sign-op-corr}
As an example of the above construction, consider the case of $\pi = \pi^\nat$.  Then the corresponding index $\ind(\diracop_{\pi^\nat})$ is equal to the signature number of $M$.  Indeed, if one takes the tensor product action $\rho \otimes \pi \otimes \bar{\pi}$ on
\begin{equation*}
  \label{eq:t-prod-conjug-space}
  C(P; V^\nat) \otimes (V^\nat)^*,
\end{equation*}
its fixed point subspace is identified to the space of the sections of $\Cliff(M)$, and the grading on the former corresponds to the left Clifford action of the volume element.  Meanwhile the fixed point subspace is canonically identified to the $\pi^\nat$-isotypic component of $C(P; V^\nat)$.
\end{example}

\section{Connections on infinite dimensional equivariant bundles}
\label{sec:conn-infin-dimens}

Now we give a more precise description of
\eqref{eq:distr-expr-rel-chern-char-naive}.  Let $A$ be the curvature $1$-form of \eqref{eq:universal-curvature-form}.  The associated connection $2$-form in $A^2(P,
\liealg{su}(N))_{\text{basic}}$ is given by
\[
\Omega_{X, Y} = A([X - A_X, Y - A_Y])
\]
for vector fields $X$ and $Y$ on $P$.

Let $\sigma$ be a section of $P \rightarrow M$ defined on an open set $U$ of
$M$.  When $X$ and $Y$ be vector fields on $M$, consider the operator
\[
\Omega^{(\sigma)}_{X, Y}.(\sigma_x, \xi_x) = (\sigma_x, \lambda(\Omega_{d
  \sigma(X_x), d \sigma(Y_x)}) \xi_x),
\]
 on the sections of $\vbdl{E}$ over $U$ as in \eqref{eq:invar-conn-action-from-section}.  This is independent of the choice of $\sigma$, and the operators $\Omega^{(\sigma)}_{X, Y}$ on different open sets naturally patch together and give a $2$-form of bundle morphisms $F_{X, Y}$ on $\vbdl{E}$.
 
 Then $F_{X, Y}$ commutes with the Clifford action, and that of $\rho \otimes \pi^\nat(\SU(N))$.  Similarly we obtain $2k$-forms of endomorphisms of $\vbdl{E}$ by
\[
F^k_{X_1,\ldots,X_{2k}} = \sum_{\sigma \in S_{2k}} (-1)^{\absolute{\sigma}} F_{X_{\sigma 1}, X_{\sigma 2}} \cdots F_{X_{\sigma (2k-1)}, X_{\sigma 2k}}.
\]

\subsection{Fell's absorption}

Let $\phi$ be a smooth function on $\SU(N)$.  Then the right convolution by
$\phi$ on $P$ defines a trace class operator $\rho_P \otimes \pi^\nat (\phi)$ on $L^2(P, V^\nat)$.  Thus one is interested in the computation of
\begin{equation}
\label{eq:test-func-trace-def}
\Tr_{L^2(\PU(N); V^\nat)}(e^{-F^{\vbdl{E}/\spinbdl}} \rho \otimes \pi^\nat(\phi)).
\end{equation}

Let us consider the following operator $S$ on $L^2(\SU(N); V^\nat)$.  When $\phi$ is a function of $\SU(N)$ into $V^\nat$, we define $S \phi$ to be the function $(S \phi)(\gamma) = \pi^\nat(\gamma) \phi(\gamma)$.  Then $S$ is an isomorphism of the $\SU(N)$-bimodule ${}_\lambda L^2(\SU(N); V^\nat)_{\rho \otimes \pi^\nat}$ into ${}_{\lambda\otimes \pi^\nat} L^2(\SU(N); V^\nat)_{\rho}$.

Then the submodule
\[
{}_\lambda L^2(\PU(N); V^\nat)_{\rho \otimes \pi^\nat} \subset {}_\lambda L^2(\SU(N); V^\nat)_{\rho \otimes \pi^\nat}
\]
is mapped to the subspace $L^2(\SU(N); V^\nat)^{\rho(\chi^\nat)}$ spanned by
\[
\ensemble{ \phi \in L^2(\SU(N); V^\nat) \suchthat \forall \gamma \in Z(\SU(N)) \colon \rho_\gamma(\phi) = \chi^\nat(\gamma) \phi}
\]
of ${}_{\lambda\otimes \pi^\nat} L^2(\SU(N); V^\nat)_{\rho}$.  This bimodule admits a direct sum decomposition
\begin{equation}
\label{eq:bimod-as-dir-sum-tens-prod}
L^2(\SU(N); V^\nat)^{\rho(\chi^\nat)} = \bigoplus_{\pi \in \widehat{\SU(N)}_\nat} {}_{\pi \otimes \pi^\nat} (V_\pi \otimes V^\nat) \otimes (V_\pi)_{\bar{\pi}}.
\end{equation}

Under the transformation $S$, the operator $\lambda(\Omega^{(\sigma)}_{X, Y})$ is identified with the operator
\begin{equation}
\label{eq:curv-action-after-trans}
\lambda \otimes \pi^\nat (\Omega^{(\sigma)}_{X, Y}) = \lambda(\Omega^{(\sigma)}_{X, Y})\otimes 1_{V^\nat} + 1 \otimes \pi^\nat (\Omega^{(\sigma)}_{X, Y}).
\end{equation}
This shows that the Clifford module $\vbdl{E}$ can be `locally' decomposed as the tensor product $\vbdl{E}_0 \otimes \spinbdl$, where $\vbdl{E}_0$ is a bundle of fiber $C(\SU(N))_{\chi^\nat}$ defined by
\[
\vbdl{E}_0 \simeq P \times_{\SU(N)} C(\SU(N))_{\chi^\nat}.
\]

\subsection{Computation of index character}

On one hand, the first part in the right hand side of \eqref{eq:curv-action-after-trans} can be thought as the action of the relative curvature form $F^{\vbdl{E}/\spinbdl}_{X, Y} = F^{\vbdl{E}_0}_{X, Y}$.  On the other hand, the second part can be thought as the action of the spin curvature $\sum_{i, j} \pairing{R^{\LevCiv}_{X, Y} e_i}{e_j} c_i c_j$.  Hence the quantity \eqref{eq:test-func-trace-def} can be translated to
\begin{equation}
\label{eq:rel-curv-straighted-formula}
\Tr_{L^2(\PU(N); V^\nat)}(\lambda(e^{-\Omega}) \rho(\phi)) = N \Tr_{L^2(\SU(N))_{\chi^\nat}}(\lambda(e^{-\Omega}) \rho(\phi)).
\end{equation}

Let $k$ be an integer.  Combining \eqref{eq:env-alg-elem-left-reg-smooth-func-right-reg-tr-pair} with \eqref{eq:rel-curv-straighted-formula}, the $2k$-form $\Tr_{\vbdl{E}}(\rho(\phi) F^k_{X_1,\ldots,X_{2k}})$ can be written as
\begin{equation*}
\sum_{\sigma \in S_{2k}, g \in Z(\SU(N))} (-1)^{\absolute{\sigma}} \pairing{T(\Omega^{(\sigma)}_{X_{\sigma 1}, X_{\sigma 2}} \cdots \Omega^{(\sigma)}_{X_{\sigma (2k-1)}, X_{\sigma 2k}})}{\lambda_g\phi}.
\end{equation*}
Consequently we obtain
\begin{equation}
\label{eq:iter-curv-form-application}
\pairing{\Tr_{\vbdl{E}}(\gamma e^{-F/\spinbdl})}{\phi} = \sum_{j=0}^{m/2} \frac{1}{j!} \sum_{g \in Z(\SU(N))} \pairing{T((\Omega^{(\sigma)})^j)}{\lambda_g\phi},
\end{equation}
which gives the meaning as a distribution on $\SU(N)$ to the expression
\eqref{eq:distr-expr-rel-chern-char-naive}.

\begin{thm}
\label{thm:induced-cliff-mod-ind-formula}
Let $\pi$ be an irreducible representation with central character $\chi^\nat$.  The distribution $\ind_{\SU(N)}(\gamma, \diracop_M)$ satisfies
\begin{equation}
  \label{eq:ind-distr-pair-char}
  \pairing{\ind_{\SU(N)}(\gamma, \diracop_M)}{\chi^\pi}
= d_\pi \int_M
  \sum_{j=0}^{m/2} \frac{1}{j!}
  \hat{A}(R_M)_{m-2j} \Tr(\pi((\Omega^{(\sigma)})^j)).
\end{equation}
\end{thm}

\begin{proof}
Let $E_\pi$ be the $\pi$-isotypic component of $\vbdl{E}$.  Since the action of $\Cliff(M)$ on $\vbdl{E}$ commutes with the action of $\SU(N)$, the vector bundle $E_\pi$ is a Clifford submodule of $\vbdl{E}$.  We show that the relative Chern character $\ch(E_\pi / \spinbdl)$ of the Clifford module $V_\pi$ is represented by $d_\pi \Tr(e^{\pi(\Omega)})$.  Then the assertion will follow from the index formula \eqref{eq:cliff-mod-ind-thm} for Clifford modules.

Let $p^\pi$ denote the projector $\int_{\SU(N)} \chi^\pi(g) g$ in the convolution algebra $L^1(\SU(N))$.  Then $E_\pi$ is the image of $\pi^P \otimes \pi^\nat(p^\pi)$.  When $\sigma$ is a section of $P$, the sections $\Gamma(\dom \sigma; E_\pi)$ of $E_\pi$ over the domain of $\sigma$ can be identified to the space of the functions 
\[
x \mapsto (\sigma_x, f_x) \quad (f_x \in \rho \otimes \pi^\nat(p^\pi).L^2(\PU(N); V^\nat))
\]
for $x \in \dom \sigma$.  By the $\SU(N)$-invariance of $\nabla$ on $\vbdl{E}$, the associated curvature form of $E_\pi$ is given by $F \rho(p^\pi)$.  

Applying the operator $S$, one sees that $\Gamma(\dom \sigma; E_\pi)$ can also be identified with the space of the functions $x \mapsto (\sigma_x, f_x)$ where $f_x \rho(p^\pi).L^2(\SU(N); V^\nat)^{\rho(\chi^\nat)}$.  By the decomposition \eqref{eq:bimod-as-dir-sum-tens-prod}, the $\SU(N)$-space $\rho(p^\pi).L^2(\SU(N); V^\nat)^{\rho(\chi^\nat)}$ is isomorphic to the direct sum of $d_\pi$ copies of $V_\pi \otimes V^\nat$.

From  \eqref{eq:curv-action-after-trans}, one sees that the action of the curvature form $F p^\pi$ on $E_\pi$ is given by $\pi(\Omega^{(\sigma)}) \otimes 1_\nat + 1_\pi \otimes \pi^\nat(\Omega^{(\sigma)})$.  The term $1_\pi \otimes \pi^\nat(\Omega^{(\sigma)})$ is the action of the spinor curvature.  Hence the relative curvature form is given by $\pi(\Omega^{(\sigma)}) \otimes 1_\nat$, which proves the assertion.
\end{proof}

\begin{rmk}
  Consider the case of $\pi = \pi^\nat$ as in Example \ref{ex:nat-comp-sign-op-corr}.  Thus the left hand side of \eqref{eq:ind-distr-pair-char} is equal to the signature number of $M$.  In the right hand side, the term $\Tr_{V^\nat}(\pi^\nat(\Omega^j))$ is equal to the $j$-th component of the relative Chern character $\ch(\wedge^* \Tang^*M/\spinbdl)$.  Hence we recover the signature formula
  \[
\sigma(M) = \int_M \hat{A}(R_M) \wedge \ch(\wedge^* \Tang^*M/\spinbdl).
\]
\end{rmk}

We obtain the following formula whose conceptual meaning is \eqref{eq:equivar-cliff-mod-ind-formula}.

\begin{cor}
\label{cor:ind-distribution-full-formula}
The distribution $\ind_{\SU(N)}(\gamma, \diracop_M)$ can be written as
\begin{equation*}
  \pairing{\ind_{\SU(N)}(\gamma, \diracop_M)}{\phi}
  = \sum_{g \in Z(\SU(N))} \int_M
  \sum_{j=0}^{m/2} \frac{1}{j!}
  \hat{A}(R_M)_{m-2j}\pairing{T((\Omega^{(\sigma)})^j)}{\lambda_g\phi},
\end{equation*}
where $\phi$ is any test function in $C^\infty(\SU(N))$.
\end{cor}

\begin{proof}
 Since both sides are invariant under conjugation for $\phi$, we may assume that $\phi$ is invariant under conjugation.  By continuity and linearity, we may assume that $\phi$ is a character of some irreducible representation of $\SU(N)$.  Then the assertion follows from Lemma \ref{lem:pu-e-nat-decomp-mult-calc} and Theorem \ref{thm:induced-cliff-mod-ind-formula}.
\end{proof}

Now, we can recover the fractional index formula \eqref{eq:mms-frac-index-formula} as a particular case of Corollary \ref{cor:ind-distribution-full-formula}.

\begin{cor}[\citelist{\cite{MR2258800}*{Theorem 2}\cite{MR2396250}*{Proposition 5 and Remark 1}}]
Let $\phi$ be a smooth function which agrees with the constant function $1$ on a neighborhood of $e$ and satisfies $\supp \phi \cap Z(\SU(N)) = \ensemble{e}$.  Then one has
\[
\pairing{\ind_{\SU(N)}(\gamma, \diracop_M)}{\phi} = \int_M \hat{A}(R_M).
\]
\end{cor}

\begin{proof}
For each $j$, the operator $T((\Omega^{(\sigma)})^j)$ is represented by a $\SU(N)$-biinvariant differential operator of degree $2j$.  Suppose that $\phi$ agrees with the constant function $1$ on a neighborhood of
$e$.  Then one has
\begin{align*}
\Tr(\rho(\phi)) &= \phi(e) = 1, & \Tr(\rho(\phi) F^j) &= \pairing{T((\Omega^{(\sigma)})^j)}{\phi} = 0 \quad (j > 0).
\end{align*}
Hence one has $\pairing{\Tr(\gamma e^{-F})}{\phi} = 1$ in
this case.  Consequently one obtains
\[
\pairing{\ind_{\SU(N)}(\gamma, \diracop_M)}{\phi} = \int_M \hat{A}(R_M),
\]
which proves the assertion.
\end{proof}

\begin{rmk}
 Since $\diracop_M$ is formulated as a transversely elliptic operator on the $\SU(N)$-manifold $P$, the Kirillov type formulation of the equivariant index formula for such operators by Berline--Vergne~\citelist{\cite{MR1369410}\cite{MR1369411}} might be also employed to prove the above result.  Our presentation is rather based on the Atiyah--Segal--Singer type formulation of the equivariant index theorem.
\end{rmk}

\paragraph{Acknowledgment} The author would like to thank R. Ponge for  
suggesting him to look at the papers of Mathai--Melrose--Singer and for  
many related discussions at the early stage of the research.  He would also like to thank S. Neshveyev, Y. Oshima, R. Nest, and Mathai V. for fruitful conversations.  He also benefited from conversations with E. Meinrenken, S. Yamamoto, R. Tomatsu, N. Ozawa, and M. Pichot.  Lastly but not least, he is deeply indebted for Y. Kawahigashi and D. E. Evans for numerous support throughout the period of the research.
\bibliography{mybibliography}

\end{document}